\documentclass[a4paper, 11pt]{article}
\usepackage[mathscr]{eucal}
\usepackage{amssymb}
\usepackage{latexsym}
\usepackage{amsthm}
\usepackage{amsmath}
\usepackage[dvips]{graphicx}
\usepackage{psfrag}
\usepackage{a4wide}

\newtheorem{theorem}{Theorem}[section]
\newtheorem{proposition}[theorem]{Proposition}
\newtheorem{lemma}[theorem]{Lemma}
\newtheorem{corollary}[theorem]{Corollary}
\theoremstyle{definition}
\newtheorem{definition}[theorem]{Definition}
\newtheorem*{constructionA}{Construction A}

\newtheorem{problem}[theorem]{Problem}

\newtheorem{example}[theorem]{Example}

\makeatletter
\@addtoreset{equation}{section}

\makeatother
\title{Combinatorial characterzations of $T$-designs in the nonbinary Johnson scheme}
\author{Hiroshi Nozaki\thanks{Department of Mathematics Education, 
	Aichi University of Education, 
	1 Hirosawa, Igaya-cho, 
	Kariya, Aichi 448-8542, 
	Japan. {\tt hnozaki@auecc.aichi-edu.ac.jp}} 
    \and Yuta Watanabe\thanks{Department of Mathematics Education, 
	Aichi University of Education, 
	1 Hirosawa, Igaya-cho, 
	Kariya, Aichi 448-8542, 
	Japan. {\tt ywatanabe@auecc.aichi-edu.ac.jp}}}

\begin{document}

\maketitle

\begin{center}
\textit{Dedicated to Professor Paul Terwilliger on the occasion of his 70th birthday.}
\end{center}

\renewcommand{\thefootnote}{\fnsymbol{footnote}}

\begin{abstract}
We study $T$-designs in the nonbinary Johnson scheme. This scheme generalizes both the Johnson and 
Hamming schemes and admits a bivariate $Q$-polynomial structure. Zhu~(2021) provided a 
combinatorial characterization of $T$-designs in this scheme for certain index sets~$T$, using 
a relationship between $T$-designs in the nonbinary Johnson scheme and relative designs in 
the nonbinary Hamming scheme. In this paper, we obtain a characterization that applies to a 
strictly larger class of index sets $T$, based on a methodological extension of Delsarte’s 
original framework (1973). This new characterization naturally recovers classical block designs and orthogonal arrays as special cases. To describe these designs uniformly, we 
introduce $(r,s)$-designs, a new family of combinatorial objects that arise naturally from our 
characterization. We also derive absolute lower bounds on the cardinality of $(r,s)$-designs 
from the multiplicities of the primitive idempotents of the nonbinary Johnson scheme, and 
construct examples with index $\lambda=1$ that attain certain natural lower bounds.
\end{abstract}

\section{Introduction}
Delsarte \cite{D73} introduced the notion of designs in association schemes, thereby providing a unified framework that incorporates classical combinatorial structures such as block designs and orthogonal arrays.
Block designs are identified with designs in the Johnson scheme, and orthogonal arrays are identified with designs in the Hamming schemes.
These association schemes are $Q$-polynomial schemes \cite{D73}, whose designs are well analyzed using the corresponding univariate orthogonal polynomials. 

For a set $T$ of natural numbers, a $T$-design in an association scheme $(X,\mathcal{R})$ is a finite subset $Y$ of the underlying set $X$ that satisfies
\[
E_i \varphi_Y = 0 \qquad \text{for each $i \in T$,}
\]
where $E_i$ is a primitive idempotent of the Bose--Mesner algebra of $(X,\mathcal{R})$ and $\varphi_Y$ is the characteristic vector of $Y$.
For $Q$-polynomial association schemes, we have a natural ordering of $\{E_i\}$, that is, there exists a polynomial $v_i$ of degree $i$ such that $v_i(E_1) = E_i$ with the entry-wise matrix product. 
This natural ordering helps to provide a good combinatorial characterization of certain association schemes. 
For instance, a $\{1,2,\ldots,t\}$-design in the Johnson scheme forms a block $t$-design, and a $\{1,2,\ldots,t\}$-design in the Hamming scheme forms an orthogonal array with strength $t$ \cite{D73}.

The nonbinary Johnson scheme is a symmetric association scheme that generalizes both the Johnson 
schemes and the Hamming schemes. 
The character tables of this scheme are written in terms of bivariate orthogonal polynomials 
\cite{Dun76, TAG85}. 
In \cite{CVZZ24}, it is proved that the nonbinary Johnson scheme, under a specific restriction, 
is a bivariate $Q$-polynomial scheme in the sense of \cite{BCPVZ24}. 
Bannai et al.\ \cite{BKZZ25} established a more general notion of multivariate 
$Q$-polynomial association schemes than that of \cite{CVZZ24}, using arbitrary monomial orders. 
Bannai et al.\ \cite{BKZZ24} further proved that the nonbinary Johnson scheme is a bivariate 
$Q$-polynomial scheme in their sense; this result does not require the restriction imposed 
in \cite{CVZZ24}. 

In pursuing a generalization of the $Q$-polynomial framework, it is natural to develop a 
combinatorial characterization of $T$-designs that is compatible with the natural ordering arising from a 
multivariate $Q$-polynomial scheme. 
As a first step in this direction, we focus on the nonbinary Johnson scheme.  
Zhu \cite{Z21} provided a combinatorial characterization of $T$-designs in the nonbinary Johnson 
schemes for specific $T$. 
This characterization is proved via the relationship between $T$-designs in the nonbinary Johnson 
schemes and relative designs in the nonbinary Hamming schemes.

In this paper, we prove a combinatorial characterization of $T$-designs via a purely 
methodological extension of Delsarte's original approach \cite{D73}, for certain $T$ that 
properly contains the specific cases considered by Zhu. 
This characterization naturally includes the notions of block designs and orthogonal arrays. 

The combinatorial property of the designs is expressed as follows. 

\begin{definition}[$(r,s)$-design]\label{def:rs-design_1}
Let $M$ be a $v\times n$ array with entries in $\{0,1,\ldots,q-1\}$ such that each row has exactly 
$w$ non-zero entries.  
Let $r,s$ be nonnegative integers with $s\leq r\leq w$. 
The matrix $M$ is called an $(r,s)$-design if for any choice of $r$ columns, the subarray obtained by taking all rows whose entries are non-zero in these $r$ columns has the following property: 
in any $s$ of these $r$ columns, every $s$-tuple from $\{1,\ldots,q-1\}$ occurs the same number of times.
\end{definition}

See Figures~\ref{fig:example-array-5col} and \ref{fig:example-array} for nontrivial examples of 
$(2,1)$-designs. 
For $q=2$, an $(r,s)$-design is equivalent to a block $r$-design for any $s$, where the rows 
are regarded as characteristic vectors of blocks. 
For $w=n$, an $(r,s)$-design is an orthogonal array of strength $s$ for any $r$. 
For $r=s$, an $(r,s)$-design coincides with the characterization of Zhu \cite{Z21}. In particular, it includes an $H$-design \cite{Hanani63} as a special case. 
Aside from these special cases, the concept of $(r,s)$-designs does not appear to have been studied 
in the literature, to the best of the author's knowledge.  
The present paper therefore opens a new direction in the study of designs.

This paper is organized as follows. 
Section~\ref{sec:2} introduces the nonbinary Johnson scheme and summarizes basic results proved in \cite{TAG85}, which will be used later. 
In Section~\ref{sec:3}, we prove a combinatorial characterization of $T$-designs in the nonbinary Johnson scheme via a methodological extension of Delsarte's original approach. 
Section~\ref{sec:4} provides several combinatorial properties of $(r,s)$-designs, which extend the combinatorial properties of the index $\lambda$ in classical block designs. 
In Section~\ref{sec:5}, we derive absolute lower bounds on the cardinality of $(r,s)$-designs from the multiplicities of the primitive idempotents of the nonbinary Johnson scheme. 
In Section~\ref{sec:6}, we construct $(r,s)$-designs with index $\lambda=1$, which are minimal designs attaining a natural lower bound.

\section{Preliminaries} \label{sec:2}
In this section, we recall the basic terminology and the fundamental results from \cite{TAG85}. 

Let $F$ be a finite set of cardinality $q$; we usually take $F = \{0,1,\ldots,q-1\}$. 
For $x \in F^n$, the (Hamming) \emph{weight} $w_H(x)$ is the number of nonzero coordinates in $x$. 
Let $W_w(n,q)$ denote the subset of $F^n$ consisting of all vectors of weight $w$. 
For $x = (x_1, \ldots, x_n)$ and $y = (y_1, \ldots, y_n)$ in $F^n$, define
\[
e(x,y) = |\{\, i \mid x_i = y_i \ne 0 \,\}| 
\quad \text{and} \quad 
n(x,y) = |\{\, i \mid x_i \ne 0,\; y_i \ne 0 \,\}|.
\]
With the relations $R_{ij}$ on $W_w(n,q)$ defined by
\[
R_{ij} = \{\, (x,y) \mid e(x,y) = w - i,\; n(x,y) = w - j \,\},
\]
the pair $J_q(w,n)=(W_w(n,q), \{R_{ij}\})$ forms a symmetric association scheme, 
which is called the \emph{nonbinary Johnson scheme} for $q\geq 3$. 

Let $\mathfrak{B}$ be the Bose--Mesner algebra of $J_q(w,n)$.  
The primitive idempotents $E_{ij}$ of $\mathfrak{B}$ are indexed by 
\[
L=\{(i,j) \mid 0\leq j \leq i \leq w, i-j\leq m\},
\]
where $m=\min\{w,n-w\}$. 
In \cite{TAG85}, the primitive idempotents are naturally obtained from certain matrices $C_{rs} \in \mathfrak{B}$, which play a combinatorial role well. 

We define the matrix $C_{rs}$ using the characters of the multiplicative abelian group $\overline{F}=F\setminus \{0\}$. 
Let $\psi_a$ $(a \in \overline{F})$ be the complex characters of $\overline{F}$, numbered in such a way that  
$\psi_a(b) = \psi_b(a)$ for all $a,b \in \overline{F}$.  
For $x, y \in F^n$ and $a \in \overline{F}$, define $\psi_a(0) = 0$ and  
\[
(x, y) = \prod_{i \in \overline{x}} \psi_{x_i}(y_i),
\]
where $\overline{x} = \{ i \mid x_i \ne 0 \}$ is the support of $x$.  
Notice that $(x, y) \ne 0$ if and only if $\overline{x} \subset \overline{y}$.

Define the \emph{multiplicative weight} $w_M(x)$ of a vector $x \in F^n$ by  
\[
w_M(x) = |\{\, i \mid x_i \ne 0,\; x_i \ne 1 \,\}|.
\]
Let 
\[W_{rs}=\{x \in F^n \mid w_H(x)=r,w_M(x)=s\}\] for $0\leq s\leq r \leq n$.
Let $\mathbb{C}(X,Y)$ be the set of matrices with rows indexed by $X$ and columns indexed by $Y$. 
Using the matrix $A_{rs}=((a,x))_{x\in W_w(n,q), a \in W_{rs}} \in \mathbb{C}(W_w(n,q),W_{rs})$, the matrix $C_{rs}$ is defined by
\[
C_{rs}=A_{rs} \overline{A_{rs}^\top}=\left(\sum_{a \in W_{rs}}(a,x)\overline{(a,y)} \right)_{x,y \in W_w(n,q)}
\]
for $(r,s) \in L$, see Equation (23) in \cite{TAG85}. 

For $(r,s)\in L$, let $\mathfrak{B}_{rs}=\operatorname{Span}_{\mathbb{C}} \{C_{kh}\}_{(k,h)\le (r,s)}$ 
be the subalgebra of the Bose--Mesner algebra $\mathfrak{B}$ spanned by the matrices 
$C_{kh}$ for all $(k,h)\in L$ such that $(k,h)\le (r,s)$ in the lexicographic order.  
The dimensions of the algebras $\mathfrak{B}_{rs}$ increase by one as $(r,s)$ 
increases in the lexicographic order.  
Define $E_{rs}\in\mathfrak{B}_{rs}$ to be the primitive idempotent that first appears 
in $\mathfrak{B}_{rs}$. Then, $\{E_{rs}\}_{(r,s) \in L}$ forms the primitive idempotents of $\mathfrak{B}$.

\section{Combinatorial characterizations of $T$-designs in $J_q(w,n)$}  \label{sec:3}
Let $\mathfrak{B}$ be the Bose--Mesner algebra of the nonbinary Johnson schemes $J_q(w,n)$. 
A matrix $M\in \mathfrak{B}$ can be expressed by the linear combination of the primitive idempotents $\{E_{ij}\}_{(i,j) \in L}$ as $M=\sum_{(i,j) \in L} c_{ij} E_{ij}$ with $c_{ij} \in \mathbb{C}$. 
The matrix $E_{ij}$ is called a {\it component} of $M$ if the corresponding coefficients $c_{ij}$ is not $0$. 
Let $\mathfrak{C}(M)$ be the set of components of $M$. 
\begin{lemma} \label{lem:component}
Let $(r,s) \in L$. 
The set of components of $C_{rs} \in \mathfrak{B}$ is 
\[
\mathfrak{C}(C_{rs})=\{E_{is}  \mid (i,s) \in L, i\leq r  \}.
\]
\end{lemma}
\begin{proof}
From Lemma 1 in \cite{TAG85}, for $(r,s),(k,h) \in L$, one has
\begin{equation}\label{eq:C_rs}
C_{rs}C_{kh}=\delta_{sh} \sum_{i=s}^{\min \{r,k\}} a_{i} C_{is},   
\end{equation}
where $a_i >0$ for each $i$. 
Since $C_{rs}$ is a positive semidefinite matrix, 
it can be expressed as a linear combination of primitive idempotents with non-negative coefficients. 
This implies that $\mathfrak{C}(C_{rs}C_{kh}) =\mathfrak{C}(C_{rs})\cap \mathfrak{C}(C_{kh})$. 
Since $E_{kh} \in \mathfrak{C}(C_{kh})$ holds, $C_{kh}$ appears in the expresion \eqref{eq:C_rs} if and only if $E_{kh}\in  \mathfrak{C}(C_{rs})$.  
Therefore, $E_{kh} \not\in \mathfrak{C}(C_{rs})$ for $s\ne h$ or $r<k$, and  $E_{kh} \in \mathfrak{C}(C_{rs})$ for $s=h$ and $s\leq k\leq r$ ($s>k$ is excluded since $(k,s)\in L$).  Thus, $\mathfrak{C}(C_{rs})$ is detemined as stated. 
\end{proof}
Let $X=W_w(n,q)$ and $\phi_Y$ the characteristic vector of $Y \subset X$. 
It is known that $Y \subset X$ is a $T$-design in $J_q(w,n) $ if and only if  
\begin{equation} \label{eq:E_ij}
E_{ij}\phi_Y=\frac{|Y|}{|X|}E_{ij}\phi_X    
\end{equation}
for each $(i,j) \in T$ \cite[Theorem 3.10]{D73}.

Let $N=\{1,2,\ldots, n\}$, $F=\{0,1,\ldots, q-1\}$, and $\overline{F}=F\setminus \{0\}$. 
Let $\mathcal{S}, \mathcal{R}$  be subsets of $N$ such that $\mathcal{S} \subset \mathcal{R}$,  $|\mathcal{R}|=r$ and $|\mathcal{S}|=s$. 
Let $\mathcal{S}=\{i_1,\ldots, i_s\}$.  
For $Y \subset X$ and $\omega=(\omega_1,\ldots, \omega_s) \in \overline{F}^s$, define 
\[
m_{\mathcal{R},\mathcal{S}}(Y,\omega)
=|\{y=(y_i)_{i \in N} \in Y  \mid \mathcal{R} \subset \overline{y}, y_{i_j}=\omega_j\, (j=1,\ldots,s)  \}|, 
\]
where $\overline{y}=\{i \mid y_i\ne 0\}$ is the support of $y\in F^n$. If $s=0$, then 
\[
m_{\mathcal{R}}(Y)=m_{\mathcal{R},\emptyset}(Y,\omega)=|\{y \in Y  \mid \mathcal{R} \subset \overline{y} \}|. 
\]
\begin{theorem}\label{thm:T-design}
Suppose $(0,0)\ne (r,s) \in L$ with $r \leq m=\min\{w,n-w\}$. 
Let 
\begin{align*}
T&=\bigsqcup_{s'\in \{0,1,\ldots,s\}}\{(i,s') \in L \mid  i \leq r\}\setminus \{(0,0)\}\\    
&= \{(i,j) \in L \mid i \leq r, j\leq s\} \setminus \{(0,0)\}. 
\end{align*}
Then, $Y$ is a $T$-design in $J_q(w,n)$ if and only if 
$m_{\mathcal{R},\mathcal{S}}(Y,\omega)$ is constant for all $\omega \in \overline{F}^s$, $\mathcal{R}$, and $\mathcal{S}$ with $\mathcal{S}\subset \mathcal{R}$, $|\mathcal{R}|=r$, and $|\mathcal{S}|=s$.  
\end{theorem} 
\begin{proof}
From Lemma~\ref{lem:component} and \eqref{eq:E_ij}, 
$Y$ is a $T$-design if and only if for each $s' \in \{0,1,\ldots,s\}$, 
\begin{equation} \label{eq:crs}
C_{rs'}\phi_Y=\frac{|Y|}{|X|} C_{rs'} \phi_X.     
\end{equation} 
Note that $(r,s') \in L$ for each $s' \in \{0,1,\ldots, s\}$ with our assumption $r\leq m$. 
    The following are equivalent conditions of \eqref{eq:crs}: 
    \begin{align*}
C_{rs'}(\phi_Y-\frac{|Y|}{|X|} \phi_X)=0 
&\Leftrightarrow A_{rs'}\overline{A_{rs'}^\top}(\phi_Y-\frac{|Y|}{|X|} \phi_X) =0\\
& \Leftrightarrow \overline{A_{rs'}^\top}(\phi_Y-\frac{|Y|}{|X|} \phi_X)=0\\
& \Leftrightarrow A_{rs'}^\top(\phi_Y-\frac{|Y|}{|X|} \phi_X)=0
\\
& \Leftrightarrow A_{rs'}^\top\phi_Y=\frac{|Y|}{|X|}A_{rs'}^\top \phi_X. 
    \end{align*}
 
    For $a \in W_{rs}$, we have $w_H(a)=|\overline{a}|=r$ and $w_M(a)=|\tilde{a}|=s$, where $\tilde{a}=\{i \mid a_i\ne 0,1\}$. 
    Let $\tilde{a}=\{i_1,\ldots, i_s\}$. 
    The $a$-th entry of $A_{rs}^\top \phi_X$ is calculated as follows: 
    \begin{align}
    (A_{rs}^\top \phi_X)(a)&=\sum_{x \in X} (a,x) \nonumber \\
    &=\sum_{x \in X} \prod_{i \in \overline{a}} 
    \psi_{a_i}(x_i)  \nonumber \\
    &= \sum_{x \in X: \overline{a} \subset \overline{x} } \prod_{i \in \overline{a}} \psi_{a_i}(x_i) \label{eq:5} \\
    &= \sum_{x \in X: \overline{a} \subset \overline{x} } \prod_{j =1}^s \psi_{a_{i_j}}(x_{i_j})  \nonumber \\
    &=(q-1)^{w-s} \binom{n-r}{w-r}\sum_{\omega_1,\ldots,\omega_s \in \overline{F}} \prod_{j =1}^s \psi_{a_{i_j}}(\omega_j) \label{eq:4} \\
    &= (q-1)^{w-s} \binom{n-r}{w-r} \prod_{j=1}^s \sum_{\omega \in \overline{F}}\psi_{a_{i_j}}(\omega)  \nonumber \\
    &=\begin{cases}
        (q-1)^{w} \binom{n-r}{w-r} &\text{if $s=0$},  \nonumber\\
        0 &\text{if $s\geq 1$ }. 
    \end{cases}
    \end{align} 
In the equality \eqref{eq:4}, we count the number of terms that take the same value, namely those in  
\[
\{ x \in X \mid \overline{a} \subset \overline{x}, x_{i_j} = \omega_j( j =1,\ldots,s) \}
\]
for given $\omega_i \in \overline{F}$. 
This is equal to 
\[
m_{\mathcal{R},\mathcal{S}}(X,\omega)=(q-1)^{w-s} \binom{n-r}{w-r}, 
\]
where $\mathcal{S} \subset \mathcal{R}$ with $|\mathcal{R}|=r$ and $|\mathcal{S}|=s$. 

(I) We first prove sufficiency.  
Fix $\mathcal{R}, \mathcal{S} \subset \{1,2,\ldots,n\}$ with $\mathcal{S} \subset \mathcal{R}$, $|\mathcal{R}| = r$, and $|\mathcal{S}| = s$.  
From the equality \eqref{eq:5}, for $a \in W_{rs'}$ ($0 \le s' \le s$) with $\overline{a} = \mathcal{R}$ and $\tilde{a} \subset \mathcal{S}=\{i_1,\ldots, i_s\}$, we have
\begin{align*}
    (A_{rs'}^\top \phi_Y)(a)
    &= \sum_{y \in Y:\, \overline{a} \subset \overline{y}} \prod_{i \in \overline{a}} \psi_{a_i}(y_i)\\
    &= \sum_{y \in Y:\, \mathcal{R} \subset \overline{y}} \prod_{j =1}^s \psi_{a_{i_j}}(y_{i_j})\\
    &= \sum_{\omega \in \overline{F}^{s}} m_{\mathcal{R},\mathcal{S}}(Y,\omega) 
       \prod_{j=1}^s \psi_{a_{i_j}}(\omega_j).
\end{align*}
Since $Y$ is a $T$-design, we have
\[
\sum_{\omega \in \overline{F}^{s}} 
   m_{\mathcal{R},\mathcal{S}}(Y,\omega) 
   \prod_{j=1}^s \psi_{a_{i_j}}(\omega_j)
   = (A_{rs'}^\top \phi_Y)(a)
   = \frac{|Y|}{|X|}(A_{rs'}^\top \phi_X)(a)
   = 
   \begin{cases}
      \dfrac{|Y|}{|X|}(q-1)^{w} \binom{n-r}{w-r}, & \text{if } s' = 0,\\[6pt]
      0, & \text{if } s' \ge 1.
   \end{cases}
\]
The number of these equalities is $|\overline{F}^{s}|$, corresponding to the choices of $a_{i_j} \in \overline{F}$ ($j=1,\ldots,s$).  
For the linear equations in the indeterminates $m_{\mathcal{R},\mathcal{S}}(Y,\omega)$ ($\omega \in \overline{F}^s$),  
the coefficient matrix 
\[
\left( \prod_{j=1}^s \psi_{a_{i_j}}(\omega_j) \right)_{(a_{i_1},\ldots,a_{i_s}) \in \overline{F}^s,\, \omega \in \overline{F}^s}
\]
is nonsingular.  
Therefore, the solution is unique, and it is indeed determined by
\[
m_{\mathcal{R},\mathcal{S}}(Y,\omega)
   = \frac{|Y|}{|X|}(q-1)^{w-s} \binom{n-r}{w-r},
\]
as can be seen by comparing with $m_{\mathcal{R},\mathcal{S}}(X,\omega)$.  
This completes the proof of sufficiency.

(II) We prove necessity. 
Assume that $m_{\mathcal{R},\mathcal{S}}(Y,\omega)$ takes a constant value $\lambda_{rs}$. 
Let $\mathcal{S}=\{i_1,\ldots, i_s\}\subset \mathcal{R}$. We have 
\[
 |\{y\in Y \mid  \mathcal{R}\subset \overline{y} \}|
 = \sum_{\omega \in \overline{F}^s}|\{y \in Y \mid  \mathcal{R} \subset \overline{y}, y_{i_j}=\omega_j\ (j=1,\ldots, s)\}|=(q-1)^s \lambda_{rs}. 
\]
By double counting the set $\{(y,\mathcal{R}) \mid  y \in Y, |\mathcal{R}|=r, \mathcal{R} \subset \overline{y}\}$, we obtain 
\[
|Y|\binom{w}{r}
=\sum_{y\in Y} |\{\mathcal{R} \mid \mathcal{R}\subset \overline{y}, |\mathcal{R}|=r  \}|
=\sum_{|\mathcal{R}|=r} |\{y\in Y \mid  \mathcal{R}\subset \overline{y} \}|
=\binom{n}{r}(q-1)^s\lambda_{rs}. 
\]
Together with $|X|=\binom{n}{w}(q-1)^w$, it follows that
\[
\lambda_{rs}=|Y|\frac{\binom{w}{r}}{\binom{n}{r}}(q-1)^{-s}
=\frac{|Y|}{|X|}\frac{\binom{n}{w}\binom{w}{r}}{\binom{n}{r}}(q-1)^{w-s}
=\frac{|Y|}{|X|}\binom{n-r}{w-r}(q-1)^{w-s}. 
\]

Moreover, $m_{\mathcal{R},\mathcal{S}'}(Y,\omega')$ is also constant for $\mathcal{S}'\subset \mathcal{R}$ with $|\mathcal{S}'|=s'\leq s$. 
Indeed, for $\mathcal{S}'=\{i_1,\ldots,i_{s'}\} \subset \mathcal{S}=\{i_1,\ldots, i_s\} \subset \mathcal{R}$, 
\begin{align}
m_{\mathcal{R},\mathcal{S'}}(Y,\omega')&=|\{y \in Y \mid \mathcal{R} \subset \overline{y},y_{i_j}=\omega_{j}'\, (j=1,\ldots,s')\}|\nonumber\\
&=\sum_{\omega \in \overline{F}^s\colon\,  \omega_{j}=\omega_{j}'\, (j=1,\ldots,s')} 
|\{y \in Y \mid \mathcal{R} \subset \overline{y},y_{i_j}=\omega_{j}\, (j=1,\ldots,s)\}|\nonumber\\
&=(q-1)^{s-s'} \lambda_{r,s} \label{eq:(q-1)ss'lambda}\\
&=\frac{|Y|}{|X|} (q-1)^{w-s'}\binom{n-r}{w-r}\nonumber\\
&=\frac{|Y|}{|X|} m_{\mathcal{R},\mathcal{S'}}(X,\omega').\nonumber 
\end{align}

It therefore follows that 
\begin{multline*}
(A_{rs'}^\top \phi_Y)(a)=\sum_{\omega' \in \overline{F}^{s'}} 
   m_{\mathcal{R},\mathcal{S'}}(Y,\omega') 
   \prod_{j=1}^{s'} \psi_{a_{i_j}}(\omega_j')\\
   =\sum_{\omega' \in \overline{F}^{s'}} 
   \frac{|Y|}{|X|}m_{\mathcal{R},\mathcal{S}'}(X,\omega') 
   \prod_{j=1}^{s'} \psi_{a_{i_j}}(\omega_j')=\frac{|Y|}{|X|}(A_{rs'}^\top \phi_X)(a),
    \end{multline*}
    which implies $Y$ is a $T$-design. 
\end{proof}
\section{Combinatorial properties of $(r,s)$-designs} \label{sec:4}

Recalling Definition~\ref{def:rs-design_1}, and taking into account the framework developed in Section~3, we adopt an equivalent formulation expressed in terms of the quantities $m_{\mathcal{R},\mathcal{S}}(Y,\omega)$.
A $v \times n$ array with entries in $F=\{0,1,\ldots,q-1\}$, where each row has exactly $w$ nonzero entries, can be viewed as a set of vectors in $F^{n}$ of weight $w$.
Since $X = W_w(n,q)$ denotes the set of all weight‑$w$ vectors in $F^{n}$, each row corresponds naturally to an element of $X$, and hence the array can be identified with a subset $Y \subset X$.
Under this identification, Definition~\ref{def:rs-design_1} may be rewritten as follows.

\begin{definition}[$(r,s)$-design]\label{def:rs-design}
Let $s$ and $r$ be non-negative integers such that $s \le r \le w$.
A subset $Y \subset X$ is called an $(r,s)$-design if the number $m_{\mathcal{R},\mathcal{S}}(Y,\omega)$ is constant for all $\omega \in \overline{F}^{s}$, and for all pairs of subsets $\mathcal{S} \subset \mathcal{R}\subset N$ with $|\mathcal{S}| = s$ and $|\mathcal{R}| = r$.
When the parameters are specified explicitly, we refer to such a design as an $(r,s)$-$(n, w, q, \lambda)$-design, where $\lambda = m_{\mathcal{R},\mathcal{S}}(Y,\omega)$.
\end{definition}

The notion of an $(r,s)$-design unifies several classical structures.
For example, when $q=2$, an $(r,s)$-$(n, w, 2, \lambda)$-design is equivalent to an $r$-$(n,w,\lambda)$ block design for any $s$.
If the elements of $Y$ are arranged as a $|Y|\times n$ array in the natural way, then the case $w=n$ shows that an $(r,s)$-$(n, n, q, \lambda)$-design corresponds to an orthogonal array $OA(|Y|, n, q-1, s)$ for any $r$.
Note that the total number of elements in $Y$ is determined by the parameters; see Corollary~\ref{cor:|Y|}.
While Definition~\ref{def:rs-design} is convenient for our analytical framework, the relationship with orthogonal arrays is more transparent when expressed in terms of Definition~\ref{def:rs-design_1}.

The concept of an $(r,s)$-design also generalizes the $H$-design introduced by Hanani \cite{E97,Hanani63}.
Let $\mathcal{G} = \{ G_1, G_2, \ldots, G_n \}$ be a collection of $n$ copies of $\overline{F} = \{1,\ldots,q-1\}$, referred to as groups.
A transverse of $\mathcal{G}$ is a subset of $G_1 \sqcup G_2 \sqcup \cdots \sqcup G_n$ that intersects each group in at most one point.
Each vector in $X = W_w(n,q)$ corresponds naturally to a $w$-element transverse of $\mathcal{G}$.
Under this identification, a subset $Y \subset X$ is an $(r,s)$-design precisely when, for every $r$-subset $\mathcal{G}'\subset \mathcal{G}$, each $s$-element transverse of $\mathcal{G}'$ is contained in a constant number of $w$-element transverses in $Y$ that intersects each group of $\mathcal{G}'$.
In particular, an $(r,r)$-$(n,w,q,1)$-design is exactly an $H$-design $H(n,q-1,w,r)$.

\begin{proposition}\label{prop:derived}
Suppose $Y \subset X$ is an $(r,s)$-$(n, w, q, \lambda)$-design.
Let $s'$ and $r'$ be non-negative integer such that $s' \le s$ and $s' \le r' \le r$.
For $\omega \in \overline{F}^{s'}$ and a pair of subsets $\mathcal{S}' \subset \mathcal{R}' \subset N$ with
$|\mathcal{S}'| = s'$, $|\mathcal{R}'| = r'$, define
\[
Y' = \{ (y_i)_{i \in N \setminus \mathcal{R}'} \mid y = (y_i)_{i \in N} \in Y, \mathcal{R}' \subset \overline{y}, y_{i_j} = \omega_j \; (j = 1,2,\ldots,s')\},
\]
where $\mathcal{S}' = \{i_1,i_2,\ldots,i_{s'}\}$, $\omega = (\omega_1, \omega_2, \ldots, \omega_{s'})$.
Then $Y'$ is an $(r-r',s-s')$-$(n-r',w-r',q,\lambda)$-design.
\end{proposition}

\begin{proof}
Let $\omega' \in \overline{F}^{s-s'}$, and let $\mathcal{S}'' \subset \mathcal{R}'' \subset N \setminus \mathcal{R}'$ with $|\mathcal{S}''| = s-s'$, $|\mathcal{R}''| = r-r'$.
Write $\mathcal{S}'' = \{i_{s'+1},i_{s'+2},\ldots,i_s\}$ and $\omega' = (\omega_{s'+1},\ldots,\omega_s)$.
Then,
\begin{align*}
m_{\mathcal{R}'',\mathcal{S}''}(Y',\omega') &= |\{y' \in Y' \mid \mathcal{R}'' \subset \overline{y'}, y'_{i_j} = \omega_j \; (j = s'+1,s'+2,\ldots,s)\}| \\
&= |\{y \in Y \mid \mathcal{R}' \cup \mathcal{R}'' \subset \overline{y}, y_{i_j} = \omega_j \; (j = 1,2,\ldots,s)\}| \\
&= \lambda.
\end{align*}
\end{proof}

\begin{proposition}\label{prop:lambda_rs}
Suppose $Y \subset X$ is an $(r,s)$-$(n, w, q, \lambda)$-design.
Let $s'$ and $r'$ be non-negative integers such that $s' \le s$ and $s' \le r' \le r$.
Then the number $m_{\mathcal{R}',\mathcal{S}'}(Y,\omega')$ is constant for all $\omega' \in \overline{F}^{s'}$, and for all pairs of subsets $\mathcal{S}' \subset \mathcal{R}'\subset N$ with $|\mathcal{S}'| = s'$ and $|\mathcal{R}'| = r'$.
Moreover, the constant is given by
\[
\lambda_{r',s'} = (q-1)^{s-s'}\frac{\binom{n-r'}{r-r'}}{\binom{w-r'}{r-r'}} \lambda.
\]
\end{proposition}

\begin{proof}
Write $\mathcal{S}' = \{i_1,i_2,\ldots,i_{s'}\}$ and $\omega' = (\omega'_1,\ldots,\omega'_{s'})$.
We count the number of elements in the following set in two different ways.
\begin{equation}\label{set}
\{ (y,\mathcal{R}) \mid y \in Y, y_{i_j} = \omega'_j \; (j = 1,2,\ldots,s'), |\mathcal{R}| = r, \mathcal{R}' \subset \mathcal{R} \subset \overline{y} \}
\end{equation}

First, there are $\binom{n-r'}{r-r'}$ ways to choose $\mathcal{R}$.
For a given $\mathcal{R}$,
take a subset $\mathcal{S}' \subset \mathcal{S} \subset \mathcal{R}$
with $|\mathcal{S}| = s$.
The computation of 
$|\{ y \in Y \mid \mathcal{R} \subset \overline{y},y_{i_j} = \omega'_j \;(j=1,\ldots,s') \}|$
is identical to the one appearing in \eqref{eq:(q-1)ss'lambda}; 
hence, by \eqref{eq:(q-1)ss'lambda}, the number of choices of $y$ is $(q-1)^{s-s'} \lambda$.
Therefore, the number of elements in the set \eqref{set} is $(q-1)^{s-s'}\binom{n-r'}{r-r'} \lambda$.

Conversely, there are $m_{\mathcal{R}',\mathcal{S}'}(Y,\omega')$ ways to choose $y$.
For a given $y$, there are $\binom{w-r'}{r-r'}$ ways to choose $\mathcal{R}$.
Therefore, the number of elements in the set \eqref{set} is $\binom{w-r'}{r-r'}m_{\mathcal{R}',\mathcal{S}'}(Y,\omega')$.

Equating both counts, we obtain $m_{\mathcal{R}',\mathcal{S}'}(Y,\omega') = \lambda_{r',s'}$.
This completes the proof.
\end{proof}

\begin{corollary}\label{cor:|Y|}
Suppose $Y \subset X$ is an $(r,s)$-$(n, w, q, \lambda)$-design.
Then the cardinality of $Y$ is given by
\[
|Y| =  (q-1)^{s} \frac{\binom{n}{r}}{\binom{w}{r}} \lambda.
\]
\end{corollary}
\begin{proof}
By Proposition~\ref{prop:lambda_rs}, we have $|Y| = \lambda_{0,0} = (q-1)^{s} \frac{\binom{n}{r}}{\binom{w}{r}} \lambda$.
\end{proof}

\begin{proposition}\label{prop:RST}
Suppose $Y \subset X$ is an $(r,s)$-$(n, w, q, \lambda)$-design.
Let $s'$, $r'$ and $t$ be non-negative integers such that $s' \le s$, $s' \le r' \le r$ and $t \le r-r'$.
For every $\omega' \in \overline{F}^{s'}$, and for all subsets $\mathcal{S}' \subset \mathcal{R}'\subset N$ and $\mathcal{T} \subset N$ with $|\mathcal{S}'| = s'$, $|\mathcal{R}'| = r'$, $|\mathcal{T}| = t$ and $\mathcal{R}' \cap \mathcal{T} = \emptyset$, we have
\[
|\{y \in Y \mid \mathcal{R}' \subset \overline{y}, y_{i_j} = \omega'_j \;(j = 1,2,\ldots,s'), \mathcal{T} \cap \overline{y} = \emptyset\}|
=
(q-1)^ {s-s'} \frac{\binom{n-r'-t}{w-r'}}{\binom{n-r}{w-r}} \lambda,
\]
where $\mathcal{S}' = \{i_1,i_2,\ldots,i_{s'}\}$ and $\omega' = (\omega'_1, \omega'_2, \ldots, \omega'_{s'})$.
\end{proposition}
\begin{proof}
We first consider the case $r' = s' = 0$.
Define
\[
\mu_t(\mathcal{T}) = |\{y \in Y \mid \mathcal{T} \cap \overline{y} = \emptyset\}|.
\]
For each $i \in N$, let
\[
Y_i = \{ y \in Y \mid i \in \overline{y}\}.
\]
For a subset $\mathcal{T}' \subset \mathcal{T}$ with $|\mathcal{T}'| = t'$,
Proposition~\ref{prop:lambda_rs} gives
\[
\left| \bigcap_{i \in \mathcal{T}'} Y_i \right| = |\{ y \in Y \mid \mathcal{T}' \subset \overline{y}\}| = \lambda_{t', 0}.
\]
By the inclusion-exclusion principle, we have
\[
\mu_t(\mathcal{T})
= \left| Y \setminus \left( \bigcup_{i \in \mathcal{T}} Y_i \right)\right| 
= \sum_{\mathcal{T}' \subset \mathcal{T}} (-1)^{|\mathcal{T}'|} \left| \bigcap_{i \in \mathcal{T}'} Y_i \right| 
= \sum_{t' = 0}^t (-1)^{t'} \binom{t}{t'} \lambda_{t',0}.
\]
Thus, $\mu_t(\mathcal{T})$ is independent of the choice of $\mathcal{T}$.
Denote this constant by $\mu_t$, and consider the set
\[
\{ (\mathcal{T}, y) \mid |\mathcal{T}| = t, y \in Y, \mathcal{T} \cap \overline{y} = \emptyset\}.
\]
We count the number of elements in this set in two different ways:

First, there are $|Y|$ choices for $y$ and for each $y$, $\binom{n-w}{t}$ choices for $\mathcal{T}$.
So the total is $|Y|\binom{n-w}{t}$.
Second, there are $\binom{n}{t}$ choices for $\mathcal{T}$ and for each $\mathcal{T}$, $\mu_t$ choices for $y$.
So the total is $\mu_t\binom{n}{t}$.
Equating both counts and applying Corollary~\ref{cor:|Y|}, we obtain 
\[
\mu_t = \frac{\binom{n-w}{t}}{\binom{n}{t}} |Y| = (q-1)^{s} \frac{\binom{n-w}{t}}{\binom{n}{t}} \frac{\binom{n}{r}}{\binom{w}{r}} \lambda
=
(q-1)^s \frac{\binom{n-t}{w}}{\binom{n-r}{w-r}} \lambda.
\]
This completes the proof for the case $r' = s' = 0$.

We now consider the case $(r',s') \neq (0,0)$.
This follows by applying the result for $r' = s' = 0$ to the derived design
\[
Y' = \{ (y_i)_{i \in N \setminus \mathcal{R}'} \mid y = (y_i)_{i \in N} \in Y, \mathcal{R}' \subset \overline{y}, y_{i_j} = \omega'_j \; (j = 1,2,\ldots,s')\},
\]
which is an $(r-r',s-s')$-$(n-r',w-r',q,\lambda)$-design by Proposition~\ref{prop:derived}.
\end{proof}
\begin{proposition}\label{prop:reduction}
Suppose $Y \subset X$ is an $(r,s)$-$(n, w, q, \lambda)$-design.
If $r-s \ge n-w$, then $Y$ can be regarded as a $(w,s)$-$(n,w,q,\lambda')$-design, where 
\[
\lambda' = \frac{1}{\binom{n-r}{w-r}} \lambda.
\]
\end{proposition}

\begin{proof}
Let $\omega \in \overline{F}^{s}$ and let $\mathcal{S} \subset \mathcal{R} \subset N$ with $|\mathcal{S}| = s$, $|\mathcal{R}| = w$.
Set $\mathcal{T} = N \setminus \mathcal{R}$
and write $\mathcal{S} = \{i_1,\ldots,i_s\}$ and $\omega = (\omega_1,\ldots,\omega_s)$.
By Proposition~\ref{prop:RST}, it follows that
\begin{align*}
m_{\mathcal{R},\mathcal{S}}(Y,\omega) &= |\{y \in Y \mid \mathcal{R} \subset \overline{y}, y_{i_j} = \omega_j \; (j = 1,2,\ldots,s)\}| \\
&= |\{y \in Y \mid \mathcal{R} = \overline{y}, y_{i_j} = \omega_j \; (j = 1,2,\ldots,s)\}| \\
&= |\{y \in Y \mid \mathcal{S} \subset \overline{y}, y_{i_j} = \omega_j \; (j = 1,2,\ldots,s), \mathcal{T} \cap \overline{y} = \emptyset\}| \\
&= \frac{1}{\binom{n-r}{w-r}} \lambda
\end{align*}
This completes the proof.
\end{proof}

According to Proposition~\ref{prop:reduction}, in an $(r,s)$-$(n, w, q, \lambda)$-design, the case $r-s \ge n-w$ is equivalent to the case $r = w$.
Therefore, it suffices to consider the case $r-s \le n-w$.
In other words, we may assume $(r,s) \in L$.

\section{Lower bounds on $(r,s)$-designs} \label{sec:5}
We recall two families of orthogonal polynomials that play a central role in the study of nonbinary Johnson schemes.
For integers $0 \le i \le n$, the Krawtchouk polynomial $K_i \in \mathbb{R}[x]$ is defined by
\[
K_i(n,q,x) = \sum_{l=0}^i (-1)^l (q-1)^{i-l}\binom{n-x}{i-l}\binom{x}{l}.
%K_s(n,q,x) = \sum_{i=0}^s (-1)^i (q-1)^{s-i}\binom{n-x}{s-i}\binom{x}{i}.
\]
Similarly, for $0 \le i \le m = \min\{w,n-w\}$, the Hahn polynomial $Q_i \in \mathbb{R}[x]$ is given by
\[
Q_i(n,w,x) = \frac{\binom{n}{i} - \binom{n}{i-1}}{\binom{w}{x}\binom{n-w}{x}}\sum_{r=0}^x (-1)^r \binom{i}{r}\binom{w-i}{x-r}\binom{n-w-i}{x-r}.
%Q_k(n,w,i) = \frac{\binom{n}{k} - \binom{n}{k-1}}{\binom{w}{i}\binom{n-w}{i}}\sum_{r=0}^k (-1)^r \binom{x}{r}\binom{w-x}{k-r}\binom{n-w-x}{k-r}.
\]
Both $K_i$ and $Q_i$ are polynomials of degree $i$.

Let
\[
K = \{(k,h) \mid 0 \le h \le k \le w,\; h \le m\}.
\]
The second eigenmatrix of the nonbinary Johnson scheme $J_q(w, n)$ can be expressed in terms of these polynomials as
\begin{equation}\label{Q}
Q_{i,j}(k,h) = \frac{\binom{n}{j}}{\binom{w}{j}}K_j(w-h,q-1,k-h)Q_{i-j}(n-j,w-j,h)
\end{equation}
for $(i,j) \in L$ and $(k,h) \in K$.
This formula is due to the result in \cite[Theorem~2]{TAG85}.
In particular, $Q_{i,j}(k,h)$ is of degree $j$ in $k$ and of degree $i$ in $h$.
The multiplicity of $J_q(w, n)$ is given by
\begin{equation}\label{m}
m_{i,j} = Q_{i,j}(0,0) = (q-2)^j \binom{n}{j} \left(\binom{n-j}{i-j} - \binom{n-j}{i-j-1}\right)
\end{equation}
for $(i,j) \in L$.
The same expression also appears in the course of the proof in \cite[Lemma~2.3]{Z21}.

\begin{lemma}\label{lem:vanish}
For $(i,j), (i',j'), (r,s) \in L$,
if $i+i' < r$ or $j+j' < s$, then the Krein number of the nonbinary Johnson scheme $J_q(w, n)$ satisfies $q_{ij,i'j'}^{rs} = 0$.
\end{lemma}
\begin{proof}
By the general property of association schemes,
(see, e.g., \cite[Theorem~2.22\,(7)]{BBIT21}), 
the second eigenmatrix given in \eqref{Q} satisfies
\[
Q_{i,j}(k,h) Q_{i',j'}(k,h) = \sum_{(r,s) \in L} q_{ij,i'j'}^{rs} Q_{r,s}(k,h)
\]
for all $(i,j), (i',j') \in L$ and $(k,h) \in K$.
Recall that $Q_{i,j}(k,h)$ is of degree $j$ in $k$ and of degree $i$ in $h$.
Hence, the product $Q_{i,j}(k,h) Q_{i',j'}(k,h)$ has degree $j+j'$ in $k$ and $i+i'$ in $h$.
On the right-hand side, each term $Q_{r,s}(k,h)$ has degree $s$ in $k$ and degree $r$ in $h$.
Therefore if $q_{ij,i'j'}^{rs} \neq 0$, we must have  $r \le i+i'$ and $s \le j+j'$.
Equivalently, if $i+i' < r$ or $j+j' < s$, then $q_{ij,i'j'}^{rs} = 0$, which proves the claim.
\end{proof}

\begin{theorem}\label{thm:Fisher}
Let $(r,s) \in L$ with $(r,s) \neq (0,0)$, and define $T = \{(i,j) \in L \mid i \le r, j \le s\} \setminus \{(0,0)\}$.
Then any $T$-design $Y$ in the nonbinary Johnson scheme $J_q(w, n)$ satisfies
\[
|Y| \ge \binom{n}{\lfloor r / 2 \rfloor} \sum_{l = 0}^{\lfloor s / 2 \rfloor} (q-2)^l \binom{\lfloor r / 2 \rfloor}{l}.
\]
\end{theorem}
\begin{proof}
Set $S = \{(i,j) \in L \mid i \le \lfloor r / 2 \rfloor, j \le \lfloor s / 2 \rfloor\}$.
Consider the matrix
\[
M = \left(\sum_{(i,j) \in S} E_{i,j}\right) \circ \left(\sum_{(i',j') \in S} E_{i',j'}\right),
\]
where $E_{i,j}$ denotes the primitive idempotent of $J_q(w, n)$ and $\circ$ denotes the entry-wise matrix product.
Since each $E_{i,j}$ has real entries, all entries of $M$ are non-negative.
Moreover, for any $x \in X$,
\[
M_{x,x} = \left(\sum_{(i,j) \in S} (E_{i,j})_{x,x} \right)^2 = \frac{1}{|X|^2}\left(\sum_{(i,j) \in S} m_{i,j} \right)^2,
\]
where $m_{i,j}$ is the multiplicity of $J_q(w, n)$, given by \eqref{m}.
Therefore 
\begin{equation}\label{M1}
\sum_{x,y \in Y} M_{x,y} \ge \sum_{x \in Y} M_{x,x} = \frac{|Y|}{|X|^2}\left(\sum_{(i,j) \in S} m_{i,j} \right)^2.
\end{equation}

On the other hand, by the definition of Krein parameters, $M$ can be expanded using Krein parameters:
\begin{align*}
M &= \sum_{(i,j) \in S} \sum_{(i',j') \in S} \left(E_{i,j} \circ E_{i',j'}\right) \\
&= \frac{1}{|X|}\sum_{(i,j) \in S} \sum_{(i',j') \in S} \sum_{(i'',j'') \in L} q_{ij,i'j'}^{i''j''} E_{i'',j''}.
\end{align*}
Thus,
\[
\sum_{x,y \in Y} M_{x,y} = \frac{1}{|X|}\sum_{(i,j) \in S} \sum_{(i',j') \in S} \sum_{(i'',j'') \in L} q_{ij,i'j'}^{i''j''} \left( \sum_{x,y \in Y} (E_{i'',j''})_{x,y}\right).
\]
Since $Y$ is a $T$-design, the inner sum vanishes whenever $(i'',j'') \in T$.
By Lemma~\ref{lem:vanish}, $q_{ij,i'j'}^{i''j''} = 0$ if $i+i' < i''$ or $j+j' < j''$.
Combine these conditions,
the only surviving term corresponds to $(i'',j'') = (0,0)$.
Therefore, by the general property of Krein parameters,
(see, e.g., \cite[Proposition~2.24\,(3)]{BBIT21}), 
\begin{equation}\label{M2}
\sum_{x,y \in Y} M_{x,y} = \frac{1}{|X|}\sum_{(i,j) \in S} \sum_{(i',j') \in S} q_{ij,i'j'}^{00} \left( \sum_{x,y \in Y} (E_{0,0})_{x,y}\right)
= \frac{|Y|^2}{|X|^2}\left(\sum_{(i,j) \in S} m_{i,j} \right).
\end{equation}

From \eqref{M1} and \eqref{M2}, we have
\[
\frac{|Y|}{|X|^2}\left(\sum_{(i,j) \in S} m_{i,j} \right)^2 \le \sum_{x,y \in Y} M_{x,y} = \frac{|Y|^2}{|X|^2}\left(\sum_{(i,j) \in S} m_{i,j} \right),
\]
which implies
\[
|Y| \ge \sum_{(i,j) \in S} m_{i,j}.
\]
By the formula for multiplicities \eqref{m}, this sum can be computed as
\begin{align*}
\sum_{(i,j) \in S} m_{i,j} &= \sum_{j = 0}^{\lfloor s / 2 \rfloor} \sum_{i = j}^{\lfloor r / 2 \rfloor} m_{ij} \\
&= \sum_{j = 0}^{\lfloor s / 2 \rfloor} (q-2)^j \binom{n}{j} \binom{n-j}{\lfloor r / 2 \rfloor-j} \\
&= \binom{n}{\lfloor r / 2 \rfloor}  \sum_{j = 0}^{\lfloor s / 2 \rfloor} (q-2)^j \binom{\lfloor r / 2 \rfloor}{j}.
\end{align*}
This completes the proof.
\end{proof}

\begin{corollary}
Let $(r,s) \in L$ with $(r,s) \neq (0,0)$ and $r \leq m$, and define $T = \{(i,j) \in L \mid i \le r, j \le s\} \setminus \{(0,0)\}$.
Then any $T$-design $Y$ in the nonbinary Johnson scheme $J_q(w, n)$ satisfies:
\begin{enumerate}
\item If $r = 2e + 1$ and $s = 2f$ for some integers $e, f$, then
\[
|Y| \ge \frac{n}{w}\binom{n-1}{e} \sum_{l = 0}^f (q-2)^l \binom{e}{l}.
\]
\item If $r = 2e + 1$ and $s = 2f+1$ for some integers $e, f$, then
\[
|Y| \ge (q-1)\frac{n}{w}\binom{n-1}{e} \sum_{l = 0}^f (q-2)^l \binom{e}{l}.
\]
\end{enumerate}
\end{corollary}
\begin{proof}
By Theorem~\ref{thm:T-design} and Definition~\ref{def:rs-design}, $Y$ is an $(r,s)$-$(n,w,q,\lambda)$-design for some $\lambda$.
\begin{enumerate}
\item Define
\[
Y' = \{(y_1,\ldots,y_{n-1}) \mid (y_1,\ldots,y_n) \in Y, y_n \neq 0\}.
\]
By Proposition~\ref{prop:derived},
$Y'$ forms an $(r-1,s)$-$(n-1,w-1,q,\lambda)$-design.
By Corollary~\ref{cor:|Y|}, we have
\[
|Y'| = (q-1)^s \frac{\binom{n-1}{r-1}}{\binom{w-1}{r-1}} \lambda = \frac{w}{n} |Y|.
\]
Applying Theorem~\ref{thm:Fisher} to $Y'$, we obtain 
\[
|Y'| \ge \binom{n-1}{e} \sum_{l = 0}^f (q-2)^l \binom{e}{l}.
\]
Combining these gives the desired inequality.
\item Define
\[
Y' = \{(y_1,\ldots,y_{n-1}) \mid (y_1,\ldots,y_n) \in Y, y_n = 1\}.
\]
By Proposition~\ref{prop:derived}, $Y'$ forms an $(r-1,s-1)$-$(n-1,w-1,q,\lambda)$-design.
By Corollary~\ref{cor:|Y|}, we have
\[
|Y'| = (q-1)^{s-1} \frac{\binom{n-1}{r-1}}{\binom{w-1}{r-1}} \lambda = \frac{w}{n(q-1)} |Y|.
\]
Applying Theorem~\ref{thm:Fisher} to $Y'$, we obtain 
\[
|Y'| \ge \binom{n-1}{e} \sum_{l = 0}^f (q-2)^l \binom{e}{l}.
\]
Combining these gives the desired inequality.
\end{enumerate}
\end{proof}

\section{Minimal $(r,s)$-designs with $\lambda=1$} \label{sec:6}
From Corollary~\ref{cor:|Y|}, we obtain the natural lower bound
\begin{equation}
|Y|
= (q-1)^{s} \frac{\binom{n}{r}}{\binom{w}{r}} \lambda
\ge
(q-1)^{s} \frac{\binom{n}{r}}{\binom{w}{r}}
\label{eq:natural_bound}
\end{equation}
for any $(r,s)$-design $Y$, and the equality is attained when $\lambda=1$.
For each fixed weight $w$, the bound~\eqref{eq:natural_bound} is stronger
than the one in Theorem~\ref{thm:Fisher} for sufficiently large $n$.
An $(r,s)$-design with $\lambda=1$ achieves the minimum possible
cardinality among all $(r,s)$-designs. 

First, we deal with two trivial cases.
If $r=s$, then $(r,r)$-designs with $\lambda=1$ are exactly $H$-designs. There are many examples of $H$-designs \cite{CaoJiZhu07,E97,HMM86,JiWuZhu05,ZhangGe10}. 
If $r=w$, then the set of supports of an $(r,s)$-design with $\lambda=1$ is the whole set of $w$-subsets of $\{1,\ldots,n\}$, which has $\binom{n}{w}$ points.
On each support, we place all rows of an $OA(v,w,q-1,s)$ with index $1$.

\begin{constructionA} We can construct an $(r,s)$-$(n, w, q, \lambda_1 \lambda_2)$-design with $v_1v_2$ vectors from combinatorial designs and orthogonal arrays as follows.
Let $\mathcal{B}$ be an $r$-$(n,w,\lambda_1)$ design with $|\mathcal{B}|=v_1$.
Let $\mathcal{A}$ be an orthogonal array $OA(v_2, w, q-1, s)$ of index $\lambda_2$.
For each block $B \in \mathcal{B}$ (regarded as a $w$-subset of the coordinate set),
we place all rows of $\mathcal{A}$ on the support $B$.
The collection of all resulting vectors forms the desired $(r,s)$-design.
\end{constructionA}

We focus on the case $\lambda_1=\lambda_2=1$ and $s<r<w<n$. 
There are many existence results for designs with index $\lambda_1=1$ (Steiner systems) \cite{BethJungnickelLenz99,ColbournMathon07} and for orthogonal arrays of index $\lambda_2=1$ \cite{HSS99,LinStufken25}. 
The following are inifinite families of $(2,1)$-designs with index $\lambda=1$. 
\begin{example}
    There exists a $(2,1)$-$(n,3,q,1)$-design for $n \equiv 1,3 \pmod{6}$ and any $q$.  
\end{example}
\begin{proof}
There exists a $2$-$(n,3,1)$-design (Steiner triple system $S(2,3,n)$) for $n \equiv 1,3 \pmod{6}$~\cite{ColbournMathon07}. 
Moreover, there exists a trivial $OA(q-1, 3, q-1, 1)$ for any $n$.  By Construction~A, we obtain the desired design. 
\end{proof}
Remark that if there exists a Steiner system of strength $t$, then we can similarly construct a $(t,1)$-design of index $\lambda=1$ with a trivial orthogonal array.  

By the same argument, we obtain an infinite family of $(3,2)$-designs with index $\lambda=2$. 
\begin{example}
    There exists a $(3,2)$-$(n,4,q,2)$-design for $n \equiv 2,4 \pmod{6}$ and any $q\ne 3,7$.  
\end{example}
\begin{proof}
There exists a $3$-$(n,4,1)$-design (Steiner quadruple system $S(3,4,n)$) for $n \equiv 2,4 \pmod{6}$ \cite{ColbournMathon07}. 
Morevere, there exists $OA((q-1)^2, 4, q-1, 2)$ for any $q$ except for $q=3,7$ \cite{BSP60,HSS99}. By Construction~A, we obtain the desired design. 
\end{proof}

In addition to the above construction, we have an example of index 1 in Figure \ref{fig:example-array-5col}. 
Figure~\ref{fig:example-array} is another non-trivial example of index 3, which is not obtained by Construction A. 
These examples are constructed by hand, and we anticipate that they admit a natural generalization. 
\begin{problem}
    Construct $(r,s)$-designs with index $\lambda=1$ which are not obtained from Construction~A. 
\end{problem}

\begin{figure}[htbp]
  \centering

  % ---- 左側 ----
  \begin{minipage}{0.48\linewidth}
    \centering
    \[
%    \begin{array}{ccccc}
%    0 & 0 & 1 & 1 & 1 \\
%    0 & 1 & 0 & 2 & 3 \\
%    0 & 2 & 3 & 0 & 2 \\
%    0 & 3 & 2 & 3 & 0 \\
%    3 & 0 & 0 & 3 & 2 \\
%    2 & 0 & 2 & 0 & 3 \\
%    1 & 0 & 3 & 2 & 0 \\
%    1 & 3 & 0 & 0 & 1 \\
%    2 & 2 & 0 & 1 & 0 \\
%    3 & 1 & 1 & 0 & 0
%    \end{array}
    \begin{array}{ccccc}
    0 & 0 & 1 & 1 & 1 \\
    0 & 1 & 0 & 2 & 2 \\
    0 & 2 & 2 & 0 & 3 \\
    1 & 0 & 0 & 3 & 3 \\
    2 & 0 & 3 & 0 & 2 \\
    3 & 3 & 0 & 0 & 1 \\
    3 & 0 & 2 & 2 & 0 \\
    2 & 2 & 0 & 1 & 0 \\
    1 & 1 & 1 & 0 & 0 \\
    0 & 3 & 3 & 3 & 0 
    \end{array}
    \]
    \caption{$(2,1)$-$(5,3,4,1)$-design with $10$ vectors}
    \label{fig:example-array-5col}
  \end{minipage}
  \hfill
  % ---- 右側 ----
  \begin{minipage}{0.48\linewidth}
    \centering
    \[
    \begin{array}{cccccc}
    0 & 0 & 1 & 1 & 1 & 1 \\
    0 & 1 & 0 & 2 & 2 & 2 \\
    0 & 1 & 1 & 0 & 1 & 2 \\
    0 & 2 & 2 & 2 & 0 & 1 \\
    0 & 2 & 2 & 1 & 2 & 0 \\
    1 & 0 & 0 & 1 & 1 & 1 \\
    1 & 0 & 2 & 0 & 2 & 2 \\
    2 & 0 & 2 & 2 & 0 & 2 \\
    2 & 0 & 1 & 2 & 2 & 0 \\
    2 & 2 & 0 & 0 & 2 & 1 \\
    2 & 2 & 0 & 1 & 0 & 2 \\
    1 & 1 & 0 & 2 & 1 & 0 \\
    1 & 1 & 1 & 0 & 0 & 1 \\
    2 & 2 & 2 & 0 & 1 & 0 \\
    1 & 1 & 1 & 1 & 0 & 0
    \end{array}
    \]
    \caption{$(2,1)$-$(6,4,3,3)$-design with $15$ vectors}
    \label{fig:example-array}
  \end{minipage}

\end{figure}

\section{Concluding remarks} 
This paper gives a combinatorial characterization of $T$-designs in the nonbinary Johnson scheme for a specific degree set $T$. 
The degree set $T$ arises from a natural ordering associated with the bivariate $Q$-polynomial property. 
We also establish the Fisher-type bounds for $T$-designs (Theorem~\ref{thm:Fisher}).

A $T$-design attaining the Fisher-type bound is called \emph{tight}. 
Tight block $t$-designs in the Johnson scheme are rare; see \cite[Section~2.2]{BBTZ17}. 
The nonexistence of tight block designs is established by showing that the zeros of certain one-variable orthogonal polynomials are not integral \cite{Bannai77,DukesSG13}. 
A similar relationship between the existence of tight designs and the zero distribution of orthogonal polynomials holds for one-variable $Q$-polynomial schemes. 
We expect that this relationship can be extended to multivariate $Q$-polynomial schemes, and aim to establish non-existence results for tight $T$-designs in the nonbinary Johnson scheme.

For classical $Q$-polynomial schemes, one of the main tools for analyzing designs is the linear programming bound based on one-variable orthogonal polynomials; see \cite[Section~6]{BBTZ17} and \cite[Theorem~5.10]{DGS77}. 
For multivariate $Q$-polynomial schemes, it is natural to expect a generalization of the linear programming bound using multivariate orthogonal polynomials.

\bigskip

\noindent
\textbf{Acknowledgments.} 
 H. Nozaki was partially supported by JSPS KAKENHI Grant Numbers JP22K03402 and JP24K06688. 
 Y. Watanabe was partially supported by JSPS KAKENHI Grant Number JP23K12953. 

 \bigskip

 \noindent 
 \textbf{Data Availability Statement.} 
This article is purely theoretical and does not involve any datasets.

\end{document}